\documentclass[11pt,reqno,final]{amsart}
\usepackage[table]{xcolor}
\usepackage{hyperref}
\usepackage{amssymb}
\usepackage{amsmath}
\usepackage{algorithmic}
\usepackage{amsthm}
\usepackage{epsfig, graphics, graphicx}
\usepackage{latexsym}
\usepackage{fullpage}
\usepackage{amsmath,amssymb,enumerate,comment}
\usepackage{accents}
\usepackage[ruled,vlined,norelsize]{algorithm2e}
\usepackage{graphicx}
\usepackage{caption}
\usepackage{subcaption}
\usepackage{color}
\usepackage[square,numbers]{natbib}
\usepackage{graphicx}
\usepackage[scaled]{helvet} 

\usepackage{cleveref}
\newtheorem{lemma}{Lemma}

\newtheorem{theorem}{Theorem}

\newtheorem{proposition}{Proposition}

\newcommand{\U}{{\boldsymbol{U}}}
\newcommand{\V}{{\boldsymbol{V}}}
\newcommand{\X}{{\boldsymbol{X}}}
\newcommand{\Id}{{\boldsymbol{Id}}}
\newcommand{\A}{{\boldsymbol{A}}}

\newcommand{\Vr}{{\boldsymbol{V}^p}}

\newcommand{\Ur}{{\boldsymbol{U}^i}}
\newcommand{\Uc}{{\boldsymbol{U}_{(j)}}}
\newcommand{\Xr}{{\boldsymbol{X}^i}}
\newcommand{\Xc}{{\boldsymbol{X}_{(j)}}}

\newcommand{\av}{\alpha_V}
\newcommand{\au}{\alpha_U}
\newcommand{\ax}{\alpha_X}
\newcommand{\conda}{{\kappa_A^2}}
\newcommand{\condu}{{\kappa_U^2}}
\newcommand{\condv}{{\kappa_V^2}}
\newcommand{\condx}{{\kappa_X^2}}

\newcommand{\bb}{{\boldsymbol{b}}}
\newcommand{\bopt}{{\boldsymbol{b}_\star}}

\newcommand{\bcur}{{\boldsymbol{b}_t}}
\newcommand{\btcur}{{\boldsymbol{\tilde{b}}_t}}
\newcommand{\bprev}{{\boldsymbol{b}_{t-1}}}

\newcommand{\x}{{\boldsymbol{x}}}
\newcommand{\xopt}{\boldsymbol{x}_\star}

\newcommand{\xcur}{\boldsymbol{x}_t}
\newcommand{\xprev}{\boldsymbol{x}_{t-1}}

\newcommand{\bbeta}{{\boldsymbol{\beta}}}
\newcommand{\betaopt}{{\boldsymbol{\beta_{\star}}}}

\newcommand{\betacur}{{\boldsymbol{\beta_{t}}}}
\newcommand{\betaprev}{{\boldsymbol{\beta_{t-1}}}}
\newcommand{\betaln}{{\boldsymbol{\beta_{LN}}}}
\newcommand{\betals}{{\boldsymbol{\beta_{LS}}}}

\newcommand{\y}{{\boldsymbol{y}}}
\newcommand{\be}{{\boldsymbol{e}}}
\newcommand{\z}{{\boldsymbol{z}}}
\newcommand{\br}{{\boldsymbol{r}}}

\newcommand{\bP}{{\boldsymbol{P}}}
\newcommand{\E}{{\mathbb{E}}}

\newcommand{\norm}[1]{\left\lVert#1\right\rVert}

\newcommand{\specialcell}[2][c]{\begin{tabular}[#1]{@{}c@{}}#2\end{tabular}}

\title{Iterative methods for solving factorized linear systems}
\author{A. Ma, D. Needell, A. Ramdas}

\begin{document}

\maketitle

\begin{abstract}
Stochastic iterative algorithms such as the Kaczmarz and Gauss-Seidel methods have gained recent attention because of their speed, simplicity, and the ability to approximately solve large-scale linear systems of equations without needing to access the entire matrix.  In this work, we consider the setting where we wish to solve a linear system in a large matrix $\X$ that is stored in a factorized form, $\X = \U\V$; this setting either arises naturally in many applications or may be imposed when working with large low-rank datasets for reasons of space required for storage. We propose a variant of the randomized Kaczmarz method for such systems that takes advantage of the factored form, and avoids computing $\X$.  We prove an exponential convergence rate and supplement our theoretical guarantees with experimental evidence demonstrating that the factored variant yields significant acceleration in convergence.
\end{abstract}


\section{Introduction}

  Recently, revived interest in stochastic iterative methods like the Kaczmarz \cite{Kac37:Angenaeherte-Aufloesung,SV09:Randomized-Kaczmarz,Nee10:Randomized-Kaczmarz,NSWjournal} and Gauss-Seidel \cite{frek,Ma2015convergence} methods has grown due to the need for large-scale approaches for solving linear systems of equations.  Such methods utilize simple projections and require access to only a single row in a given iteration, hence having a low memory footprint.  For this reason, they are very efficient and practical for solving extremely large, usually highly overdetermined, linear systems.  In this work, we consider algorithms for solving linear systems when the matrix is available in a factorized form.  As we discuss below, such a factorization may arise naturally in the application, or may be constructed explicitly for efficient storage and computation.  We seek a solution to the original system directly from its factorized form, without the need to perform matrix multiplication.

	To that end, borrowing the notation of linear regression from statistics, suppose we want to solve the linear system $\X \bbeta = \y$ with $\X \in \mathbb{C}^{m \times n}$.
	However, instead of the full system $\X$, we only have access to $\U,\V$ such that $\X = \U\V$. 
	In this case, we want to solve the linear system: 
		\begin{equation}
			\U \V \bbeta = \y,
			\label{eq:fullsys}
	\end{equation}
	where $\U \in \mathbb{C}^{m \times k}$ and $\V \in \mathbb{C}^{k \times n}$. 
	Instead of taking the product of $\U$ and $\V$, to form $\X$, which may not be desirable, we approach this problem using stochastic iterative methods to solve the individual subsystems
		\begin{align}
			\U \x &= \y \label{eq:subsys1}\\
			\V \bbeta &= \x \label{eq:subsys2},
		\end{align}
in an alternating fashion.
	Note that $\bbeta$ in \eqref{eq:subsys2} is the vector of unknowns that we want to solve for in \eqref{eq:fullsys} and $\y$ in \eqref{eq:subsys1} is the known right hand side vector of \eqref{eq:fullsys}. If we substitute \eqref{eq:subsys2} into \eqref{eq:subsys1}, we acquire the full linear system \eqref{eq:fullsys}. We will often refer to \eqref{eq:fullsys} as the ``full system'' and \eqref{eq:subsys1} and \eqref{eq:subsys2} as ``subsystems'', and say that a system is consistent if it has at least one solution (and inconsistent otherwise). 
	
There are some situations when approximately knowing $\x$ would suffice. We assume that (for reasons of interpretability, or for downstream usage) the scientist is genuinely interested in solving the full system, i.e. she is interested in the vector $\bbeta$, not in $\x$. 

It is arguably of practical interest to give special importance to the case of $k < \min(m,n)$, which arises in modern data science as motivated by  the following examples, but we discuss other settings later.

	\subsection{Motivation} 
		If $\X$ is large and low-rank, one may have many reasons to work with a factorization of $\X$. We shall discuss three reasons below --- algorithmic, infrastructural, and statistical.
		
		Consider data matrices encountered in ``recommender systems'' in machine learning \cite{adomavicius2005toward, koren2009matrix, ma2011recommender, takacs2008investigation, verbert2012context}. For concreteness, consider the Netflix (or Amazon, or Yelp) problem, where one has a users-by-movies matrix whose entries correspond to ratings given by users  to movies. $\X$ is usually quite well approximated by low-rank matrices --- intuitively, many rows and columns are redundant because every row is usually similar to many other rows (corresponding to users with similar tastes), and every column is usually similar to many other columns (corresponding to similar quality movies in the same genre). Usually we have observed only a few entries of $\X$, and wish to infer the unseen ratings in order to provide recommendations to different users based on their tastes. Algorithms for ``low-rank matrix completion'' have proved  to be quite successful in the applied and theoretical machine learning community \cite{candes2009exact, koltchinskii2011nuclear, wright2009robust, keshavan2010matrix}. One popular algorithm,  alternating-minimization \cite{jain2013low}, chooses a (small) target rank $k$, and tries to find $\U,\V$ such that $\X_{ij} \approx (\U\V)_{ij}$ for all the observed entries $(i,j)$ of $\X$. As its name suggests, the algorithm alternates between solving for $\U$ keeping $\V$ fixed and then solving for $\V$ keeping $\U$ fixed. In this case, at no point does the algorithm even form the entire completed (inferred) matrix $\X$, and the algorithm only has access to factors $\U,\V$ simply due to \textit{algorithmic} choices. 

There may be other instances where a data scientist may have access to the full matrix $\X$, but in order to reduce the memory storage footprint, or to communicate the data, may explicitly choose to decompose $\X \approx \U\V$ and discard $\X$ to work with the smaller matrices instead. 

		Consider an example motivated by ``topic modeling'' of text data. 
		Suppose Google has scraped the internet for English documents (or maybe a subset of documents like news articles), to form a document-by-word matrix $\X$, where each entry of the matrix indicates the number of times that word occurred in that document. Since many documents could be quite similar in their content (like articles about the same incident covered by different newspapers), this matrix is easily seen to be low-rank. This is a classic setting for applying a machine learning technique called ``non-negative matrix factorization'' \cite{lee2001algorithms, xu2003document, ma2014improving}, where one decomposes $\X$ as the product of two low-rank non-negative matrices $\U,\V$; the non-negativity is imposed for human interpretability, so that $\U$ can be interpreted as a documents-by-topics matrix, and $\V$ as a topics-by-words. In this case, we do not have access to $\X$ as a result of \textit{systems infrastructure} constraints (memory/storage/communication).

Often, even for modestly sized data matrices, the relevant ``signal'' is contained in the leading singular vectors corresponding to large singular values, and the tail of small singular values is often deemed to be ``noise''. This is precisely the idea behind the classical topic of principal component analysis (PCA), and the modern machine learning literature has proposed and analyzed a variety of algorithms to approximate the top $k$ left and right singular vectors in a streaming/stochastic/online fashion \cite{goes2014robust}. Hence, the factorization may arise from a purely \textit{statistical} motivation.

Given a vector $\y$ (representing age, or document popularity, for example), suppose the data scientist is interested in regressing $\X$ onto $\y$, for the purpose of scientific understanding or to take future actions. Can we utilize the available factorization efficiently, designing methods that work directly on the lower dimensional factors $\U$ and $\V$ rather than computing the full system $\X$?

		\textit{Our goal will be to propose iterative methods that work directly on the factored system, eliminating the need for a full matrix product and potentially saving computations on the much larger full system.}
	
	\subsection{Main contribution}
		We propose two stochastic iterative methods for solving system \eqref{eq:fullsys} without computing the product of $\U$ and $\V$. 
		Both methods utilize iterates of well studied algorithms for solving linear systems.
		When the full system is consistent, the first method, called RK-RK, interlaces iterates of the Randomized Kaczmarz (RK) algorithm to solve each subsystem and finds the optimal solution. 
		When the full system is inconsistent, we introduce the REK-RK method, an interlacing of Randomized Extended Kaczmarz (REK) iterates to solve \eqref{eq:subsys1} and RK iterates to solve \eqref{eq:subsys2}, that converges to the so-called ordinary least squares solution.  We prove linear (``exponential'') convergence to the solution in both cases.
		
	\subsection{Outline}
		In the next section, we provide background and discuss existing work on stochastic methods that solve linear systems. 
		In particular, we describe the RK and REK algorithms as well as the Randomized Gauss-Seidel (RGS) and Randomized Extended Gauss-Seidel (REGS) algorithms. 
		In Section~\ref{sec:rpvar} we investigate variations of settings for subsystems \eqref{eq:subsys1} and \eqref{eq:subsys2} that arise depending on the consistency and size of $\X$.
		Section~\ref{sec:main} introduces our proposed methods, RK-RK and REK-RK. We provide theory that shows linear convergence in expectation to the optimal solution for both methods.
		Finally, we present experiments in Section~\ref{sec:exp} and conclude with final remarks and future work in Section~\ref{sec:conclu}.
		
		\subsection{Notation}
		Here and throughout the paper, matrices and vectors are denoted with boldface letters (uppercase for matrices and lowercase for vectors). 
		We call $\Xr$ the $i^{th}$ row of the matrix $\X$ and $\Xc$ the $j^{th}$ column of $\X$.
		The euclidean norm is denoted by $\norm{\cdot}_2$ and the Frobenius norm by  $\norm{\cdot}_F$.  Lastly, $\X^*$ denotes the adjoint (conjugate transpose) of the matrix $\X$. Motivated by applications, we allow $\X$ to be rank deficient and assume that $\U$ and $\V$ are full rank.


\section{Background and Existing Work}
		In this section we summarize existing work on stochastic iterative methods and different variations of linear systems. 
		
		\subsection{Linear Systems}
		Linear systems take on one of three settings determined by the size of the system, rank of the matrix $\X$, and the existence of a solution. First we discuss solutions to systems with full rank matrices $\X$ then remark on how rank deficiency affects the desired solution.

		 In the full rank underdetermined case, $m < n$ and the system has infinitely many solutions; here, we often want to find the least Euclidean norm solution to \eqref{eq:fullsys}:
		\begin{equation}
			\betaln := \X^*(\X \X^*)^{-1} \y .
		\end{equation}
		Clearly, $\X \betaln = \y$, and all other solutions to an underdetermined system can be written as $\bb = \betaln + \z$ where $\X \z = \textbf{0}$.
		
		In the overdetermined setting, we have $m > n$ and the system can have a unique (exact) solution or no solution. If there is a unique solution, the linear system is called an overdetermined \textit{consistent} system.
		When $\X$ is full rank, the optimal unique solution is $\bbeta_{\text{uniq}}$ such that $\X \bbeta_{\text{uniq}} = \y$:
		\begin{equation}
			\bbeta_{\text{uniq}} := (\X^* \X)^{-1}\X^* \y. 
		\end{equation}
		 If there is no exact solution, the system is called an overdetermined \emph{inconsistent} system. When a system is inconsistent and $\X$ is full rank, we often seek to minimize the sum of squared residuals, i.e. to find the ordinary least squares solution
		\begin{equation}
			\betals := (\X^* \X)^{-1}\X^* \y. 
			\label{eq:overincon}
		\end{equation}
		The residual can be written as $\br = \X \betals - \y$. Note that $\X^*\br = \textbf{0}$, which can be easily seen by substituting $\y = \X \betals + \br$ into \eqref{eq:overincon}. For simplicity, we will refer to the matrix $\X$ of a linear system as consistent or inconsistent when the system itself is consistent or inconsistent.
		
		If the matrix $\X$ in the linear system $\X \bbeta =\y$ is rank deficient, then there are infinitely many solutions to the system regardless the size of $m$ and $n$. In this case, we again want the least norm solution in the underdetermined case and the ``least-norm least-squares" solution in the overdetermined case,
		\begin{equation}
		\betaln :=
\begin{cases}
		\X^*(\X \X^*)^\dagger \y , & \text{if $m < n$ }      \\
        (\X^* \X)^\dagger\X^* \y, & \text{if $m > n$ },
\end{cases}
\label{eq:lowrank}
		\end{equation}
where $(\cdot)^\dagger$ is the pseudo-inverse. General solutions to the linear system can be written as $\bb = \betaln + \z$ where $\X \z = \textbf{0}$---note that $\y = \X \bb = \X \betaln + \X \z = \X \betaln$. Similar to the full rank case, when the low-rank system is inconsistent, we can write $\y = \X \betaln + \br$, again where $\X^*\br = \textbf{0}$. 

		\subsection{Randomized Kaczmarz and its Extension}
		\label{sec:existing_rk}
			The Kaczmarz Algorithm \cite{Kac37:Angenaeherte-Aufloesung} solves a linear system $\X \bbeta = \y$ by cycling through rows of $\X$ and projects the estimate onto the solution space given by the chosen row. 
			It was initially proposed by Kaczmarz \cite{Kac37:Angenaeherte-Aufloesung} and has recently regained interest in the setting of computer tomography where it is known as the Algebraic Reconstruction Technique \cite{GBH70:Algebraic-Reconstruction,Nat01:Mathematics-Computerized,Byr08:Applied-Iterative,herman2009fundamentals}. 
			The randomized variant of the Kaczmarz method introduced by Strohmer and Vershynin \cite{SV09:Randomized-Kaczmarz} was proven to converge linearly in expectation for consistent systems.
			Formally, given $\X$ and $\y$ of \eqref{eq:fullsys}, RK chooses row $i \in \{1, 2, ... m \}$ of $\X$ with probability $\frac{\norm{\Xr}^2_2}{\norm{\X}^2_F}$, 
			and projects the previous estimate onto that row with the update
			\begin{equation*}
				\betacur := \betaprev + \frac{(\y_i - \Xr \betaprev)}{\norm{\Xr}^2_2} (\Xr)^*.
			\end{equation*}
			Needell \cite{Nee10:Randomized-Kaczmarz} later studied the inconsistent case and showed that RK does not converge to the least squares solution for inconsistent systems, but rather converges linearly to some convergence radius of the solution.
			To remedy this, Zouzias and Freris \cite{ZF12:Randomized-Extended} proposed the Randomized Extended Kaczmarz (REK) algorithm to solve linear systems in all settings. 
						For REK, row $i\in \{1, 2, ... m \}$ and column $j \in \{1, ...n \}$ of $\X$ are chosen at random with probability 			
			\begin{equation*}
				\bP(row = i) = \frac{\norm{\Xr}^2_2}{\norm{\X}^2_F} ~,~  \quad \bP(column = j) = \frac{\norm{\Xc}^2_2}{\norm{\X}^2_F},
			\end{equation*}
			and starting from $\bbeta_0=\textbf{0}$ and $\z_0=\y$, every iteration computes 
			\begin{equation*}
			\betacur := \betaprev + \frac{(\y^i -\z^i_t- \Xr \betaprev )}{\|\Xr\|^2_2} (\Xr)^*, \quad \z_{t} := \z_{t-1} - \frac{\langle \Xc, \z_{t-1} \rangle}{\|\Xc\|_2^2}\Xc.				
			\end{equation*}
			REK finds the optimal solution in all linear system settings. 
			In the consistent setting, it behaves as RK.
			In the \textit{overdetermined} inconsistent setting, $\z$ estimates the residual vector $\br $ and allows $\betacur$ to converge to the true least squares solution of the system. 
			REK was shown to converge linearly in expectation to the least-squares solution by Zouzias and Freris \cite{ZF12:Randomized-Extended}.

		\subsection{Randomized Gauss-Seidel and its Extension}
		The Gauss-Seidel method was originally published by Seidel but it was later discovered that Gauss had studied this method in a letter to his student \cite{Ben09:Key-Moments}. 
		Instead of relying on rows of a matrix, the Gauss-Seidel method relies on columns of $\X$.
		The randomized variant was studied by Leventhal and Lewis \cite{LL10:Randomized-Methods} shortly after RK was published. 
		The randomized variant (RGS) requires a column $j$ to be chosen randomly with probability $\frac{\norm{\Xc}^2_2}{\norm{\X}^2_F}$,
	and updates at every iteration
		\begin{equation}\label{eq:rgs}
			\betacur := \betaprev + \frac{\Xc^*(\y-\X \betaprev )}{\|\Xc\|^2_2} \be_{(j)},
		\end{equation}
		where $\be_{(j)}$ is the $j^{th}$ basis vector (a vector with 1 in the $j^{th}$ position and 0 elsewhere). Leventhal and Lewis \cite{LL10:Randomized-Methods} showed that RGS converges linearly in expectation when $\X$ is overdetermined.
		However, it fails to find the least norm solution for an \textit{underdetermined} linear system \cite{Ma2015convergence}. 
				The Randomized Extended Gauss-Seidel (REGS) resolves this problem, much like REK did for RK in the case of \textit{overdetermined} systems. 
		The method chooses a random row and column of $\X$ exactly as in REK, and then updates at every iteration
		 \begin{align*}
		 	\betacur &:= \betaprev + \frac{\Xc^*(\y-\X \betaprev )}{\|\Xc\|^2_2} \be_{(j)} ~,\\
			\bP_i &:= \Id_n - \frac{(\Xr)^*\Xr}{\|\Xr\|_2^2} ~,\\
			\z_t &:= \bP_i(\z_{t-1} + \betacur - \betaprev), 
		\end{align*}
		and at any fixed time $t$, outputs $\betacur - \z_t$ as the estimated solution to $\X \bbeta = \y$. Here, $\Id_n$ denotes the $n \times n$ identity matrix. This extension works for all variations of linear systems and was proven to converge linearly in expectation by Ma et al.~\cite{Ma2015convergence}.

				The RK and RGS methods along with their extensions are extensively studied and compared in \cite{Ma2015convergence}.
		Table~\ref{tab:rkrgs} summarizes the convergence properties of each of the randomized methods and their extensions.

				\begin{table}[h!]
		\centering
		\begin{tabular}{|l|l|l|l|}
			\hline
			Method & \specialcell{Overdetermined,\\ consistent : \\ convergence to $\bbeta_{\text{uniq}}$? } & \specialcell{Overdetermined,\\ inconsistent : \\ convergence to $\betals$?} & \specialcell{Underdetermined :\\ convergence to $\betaln$?} \\
			\hline
			RK & Yes \cite{SV09:Randomized-Kaczmarz} & No \cite{Nee10:Randomized-Kaczmarz} & Yes \cite{Ma2015convergence}\\
			\hline
			REK & Yes \cite{ZF12:Randomized-Extended} & Yes \cite{ZF12:Randomized-Extended} & Yes \cite{Ma2015convergence}\\
			\hline
			RGS & Yes \cite{LL10:Randomized-Methods} & Yes \cite{LL10:Randomized-Methods} & No \cite{Ma2015convergence}\\
			\hline
			REGS & Yes \cite{Ma2015convergence} & Yes \cite{Ma2015convergence} & Yes \cite{Ma2015convergence} \\
			\hline
			\end{tabular}
			\caption{Summary of convergence properties of randomized methods under all settings.} 
			\label{tab:rkrgs}
		\end{table}

		In this paper, we focus on using combinations of RK and REK but also discuss RGS and REGS for comparison.
		We choose to focus on RK and REK because their updates consist only of scalar operations and inner products as opposed to REGS which requires an outer product. The methods proposed are easily  extendable to RGS and REGS.


\section{Variations of Factored Linear Systems}
\label{sec:rpvar}
Our proposed methods rely on interleaving solution estimates to subsystem \eqref{eq:subsys1} and subsystem \eqref{eq:subsys2}.  
	Because the convergence of RK, RGS, REK, and REGS are heavily dependent on the number of rows and columns in the linear system, it is important to discuss how the settings of \eqref{eq:subsys1} and  \eqref{eq:subsys2} are determined by $\X$.
	In this section, we will discuss when we can expect our methods to solve the full system. 

\begin{table}[h]
\centering
\begin{tabular}{|c|c|}
\hline
\multicolumn{1}{|c|}{Linear System} & \multicolumn{1}{c|}{Optimal Solution} \\ \hline
$\X \bbeta =  \y $ \eqref{eq:fullsys}                   & $\betaopt$                            \\ \hline
$\U \x = \y$     \eqref{eq:subsys1}                   & $\xopt$                               \\ \hline
$\V \bb = \x$       \eqref{eq:subsys2}                  & $\bopt$                               \\ \hline
\end{tabular}
\caption{Summary of notation for linear systems discussed and their solutions}
\label{tab:systems}
\end{table}
	
	For simplicity in notation, we will denote  $\betaopt$, $\xopt$, and $\bopt$ as the ``optimal'' solution of \eqref{eq:fullsys}, \eqref{eq:subsys1}, and \eqref{eq:subsys2} respectively, as summarized in Table~\ref{tab:systems}.
	By ``optimal" solution for \eqref{eq:subsys1} and \eqref{eq:subsys2}, we mean the unique, least norm, or the least squares solution, depending on the type of system (overdetermined consistent, underdetermined, overdetermined inconsistent). Since we assume that $\U\V$ may be low-rank, $\betaopt$ is going to be the least norm solution as described in \eqref{eq:lowrank}.
	Table~\ref{tab:underover} presents such a summary depending on the size of $k$ with respect to $m$ and $n$. 

\begin{table}[h]
\centering
\scalebox{.9}{
\begin{tabular}{|c|ll|ll|ll|}
\hline
$\X$                                                              & \multicolumn{2}{l|}{$k < \min\{m,n\} $}                                          & \multicolumn{2}{l|}{$\min\{m,n\} < k < \max \{m,n\}$}   &   \multicolumn{2}{l|}{$k > \max\{m,n\}$}                        \\ \hline
Underdetermined                                                   & \begin{tabular}[c]{@{}l@{}} $\U$ = Over, Consis.\\ $\V$ = Under\end{tabular} & (S1)  & \begin{tabular}[c]{@{}l@{}}\cellcolor{black!25}$\U$ = Under \\ \cellcolor{black!25}$\V$ = Under\end{tabular}   \cellcolor{black!25}              & \cellcolor{black!25}(S2)  & \begin{tabular}[c]{@{}l@{}} \cellcolor{black!25}$\U$ = Under\\\cellcolor{black!25}$\V$ = Over, (In)con.\end{tabular} & \cellcolor{black!25} (S2) \\ \hline
\begin{tabular}[c]{@{}l@{}}Overdetermined\\ Consis.\end{tabular}  & \begin{tabular}[c]{@{}l@{}}$\U$ = Over, Consis.\\ $\V$ = Under\end{tabular} & (S1)  & \begin{tabular}[c]{@{}l@{}}$\U$ = Over, Consis.\\ $\V$ = Over, Consis.\end{tabular} & (S1)  & \begin{tabular}[c]{@{}l@{}} \cellcolor{black!25}$\U$ = Under\\\cellcolor{black!25}$\V$ = Over, (In)con. \end{tabular} & \cellcolor{black!25} (S2) \\ \hline
\begin{tabular}[c]{@{}l@{}}Overdetermined\\ Inonsis.\end{tabular} & \begin{tabular}[c]{@{}l@{}} $\U$ = Over, Incon.\\ $\V$ = Under\end{tabular}  & (S3b) & \begin{tabular}[c]{@{}l@{}}\cellcolor{black!25}$\U$ = Over, Incon.\\ \cellcolor{black!25}$\V$ = Over, (In)con.\end{tabular}  &\cellcolor{black!25}(S3a) & \begin{tabular}[c]{@{}l@{}} \cellcolor{black!25}$\U$ = Under\\\cellcolor{black!25}$\V$ = Over, (In)con.\end{tabular} & \cellcolor{black!25} (S2) \\ \hline
\end{tabular}
}

\caption{Summary of types of matrices $\U$ and $\V$ for given $m, n,$ and $k$ relations. The first column indicates the setting where $k$ is less than both $m$ and $n$, the second column when $k$ is between $m$ and $n$, and the third when $k$ is greater than both $m$ and $n$. The cells in white indicate when our proposed methods will converge to the solution of the full system, gray cells indicate when our methods will not. Recall that $\X \in \mathbb{C}^{m \times n}$, $\U \in \mathbb{C}^{m \times k}$ and $\V \in \mathbb{C}^{k \times n}$. We denote with ``(In)con" systems that can be either consistent or inconsistent. Arguably, the $k < \min\{m,n\}$ setting is most practically relevant, and in this case our methods do indeed recover the optimal solution. 
}
\label{tab:underover}
\end{table}	

 We spend the rest of this section justifying the shading in Table~\ref{tab:underover}. 	For this, we split Table~\ref{tab:underover} into three scenarios : (S1) $\U$ is overdetermined and consistent, (S2) $\U$ is underdetermined, (S3a and S3b) $\X$ is overdetermined and inconsistent. It should be noted that in Scenarios S1 and S2, the overdetermined-ness or underdetermined-ness of each subsystem follows immediately from sizes of $m$, $n$, and $k$ and the assumption that the subsystems are full rank. We use the over/underdetermined-ness of each subsystem to show $\bopt = \betaopt$ for Scenario S1 and $\bopt \neq \betaopt$ for Scenario S2. In Scenario S3, a little more work needs to be done to conclude the consistency of each subsystem. For Scenario S3a and S3b, we first investigate how inconsistency in \eqref{eq:fullsys} affects the consistency of \eqref{eq:subsys1} and \eqref{eq:subsys2}, then show that for Scenario S3a, $\bopt \neq \betaopt$ and for Scenario S3b $\bopt = \betaopt$.  This section provides the intuition on when we should expect our methods (or similar ones based on interleaving solutions to the subsystems) to work. However, one may also skip ahead to the next section where formally we present our algorithm and main results.

	\begin{itemize}
	
		\item \textbf{Scenario S1: U overdetermined, consistent.}
			When $\U$ is overdetermined and consistent, we find that solving \eqref{eq:subsys1} and \eqref{eq:subsys2} gives us the optimal solution of \eqref{eq:fullsys}. 
			Indeed, in the case where $\V$ is overdetermined and consistent, we have:
			\begin{align*}
						\bopt &= (\V^*\V)^{-1}\V^* (\U^*\U)^{-1}\U^* \y \\
								  &= (\V^*\V)^{-1}\V^* (\U^*\U)^{-1}\U^* \U \V \betaopt \\
								  & = (\V^*\V)^{-1}\V^* \V \betaopt \\
								  & = \betaopt 
			\end{align*}
			In the case where $\V$ is underdetermined, we have:
			\begin{align*}
						\V \bopt &= \V \V^*(\V \V^*)^{-1} (\U^*\U)^{-1}\U^* \y  \\
						         &= \V \V^*(\V \V^*)^{-1} (\U^*\U)^{-1}\U^* \U \V \betaopt  \\
						         &= \V \betaopt
			\end{align*}
				
			Since $\X$ is possibly low-rank, we still need to argue that this implies that $\bopt = \betaopt$, i.e. $\bopt$ is indeed the least norm solution to the full system. Suppose towards a contradiction that $\betaopt$ is the least norm solution of the full system but not subsystem~\eqref{eq:subsys2}; in other words assume that $\betaopt = \bopt + \bb$ where $\V \bb = \textbf{0}$ and $\bb$ is nontrivial. Multiplying both sides by $\X$, we see that
			since $\betaopt$ has a nontrivial component $\bb$ such that $\X \bb = \textbf{0}$, it  cannot be the least norm solution to the full system as assumed, reaching a contradiction. 
		\textit{Therefore, in the consistent case when $\U$ is overdetermined, we have that $\bopt = \betaopt$, and may hope that our proposed methods will be able to solve the full system \eqref{eq:fullsys} utilizing the subsystems \eqref{eq:subsys1} and \eqref{eq:subsys2}.}
		\item \textbf{Scenario S2: U underdetermined.}
			When $\U$ is underdetermined, solving \eqref{eq:subsys1} and \eqref{eq:subsys2} for their optimal solutions does not guarantee the optimal solution of the full system. 
			Intuitively, \eqref{eq:subsys1} has infinitely many solutions and $\xopt = \x_{LN} \not = \V \betaopt$.
			Mathematically, investigating $\bopt$, we find that 
			\begin{align*}
				\V \bopt &= \V \V^*(\V \V^*)^{-1} \U^*(\U \U^*)^{-1} \y  \\
								&= \V \V^*(\V \V^*)^{-1} \U^*(\U \U^*)^{-1} \U \V\betaopt \\
								&= \U^*(\U \U^*)^{-1} \U \V\betaopt \\
								& \not = \V \betaopt \text{  if  $\V$ underdetermined } \\
					\bopt & = (\V^*\V)^{-1}\V^* (\U^*(\U \U^*)^{-1} \y - \br_V)\\
							  &= (\V^*\V)^{-1}\V^* \U^*(\U \U^*)^{-1} \U \V \betaopt  \\
							  & \not = \betaopt \text{  if  $\V$ overdetermined, inconsistent, } 
			\end{align*}
			where we rewrite $\xopt = \V \bopt + \br_V $ for $\br_V \in null(\V^*)$ since subsystem \eqref{eq:subsys2} may be inconsistent. Note that if subsystem \eqref{eq:subsys2} is consistent then we simply have $\br_V = 0$ and the above calculation still carries through.
\textit{Therefore, we do not expect our proposed methods to succeed when $\U$ is underdetermined.}  Fortunately, this case seems to be of little practical interest, since factoring an underdetermined system does not typically save any computation. 
		\item \textbf{Scenario S3: X inconsistent.} 
			Before we discuss whether it's possible to recover the optimal solution to the full system, we must first discuss what $\X$ being inconsistent implies about the subsystems \eqref{eq:subsys1} and \eqref{eq:subsys2}.
			In particular, one needs to determine whether inconsistency in the full system creates inconsistencies in the individual subsystems. 
			If $\X$ is inconsistent then we have $\X \betaopt + \br = \y$ where $\betaopt$ is the optimal solution of \eqref{eq:fullsys} and $\X^* \br = \textbf{0}$.
			Now, consider decomposing $\br = \br_1 + \br_2$ where $\U^* \br_1 = \textbf{0}$, $\U^* \br_2 \neq \textbf{0}$, and $\V^*\U^*\br_2 = 0$. 
			Notice that $\X^* \br = \V^*\U^*(\br_1 + \br_2) = \V^*\U^*\br_1 + \V^*\U^*\br_2 = \textbf{0}$, as desired. 
			We want to decompose the full system $\X \betaopt + \br = \y$ into two subsystems.
			Following a similar thought process as before, we choose to decompose our full system into the following: 
			\begin{align}
				& \U \x + \br_1 + \br_2= \y \label{eq:sub1incon} \\
				& \V \bb = \x. \label{eq:sub2incon}
			\end{align}

			Clearly, \eqref{eq:sub1incon} is inconsistent since $\U^*\br_1=\textbf{0}$ and $\U^* \br_2 \neq \textbf{0}$. 
			
			Because $\U$ must be overdetermined for $\X$ to be inconsistent, $\xopt = (\U^* \U)^{-1}\U^*(\y - \br_1 - \br_2) = (\U^* \U)^{-1}\U^*(\y -  \br_2)$ is the least squares solution to \eqref{eq:sub1incon}.
			\begin{itemize}
				
				\item \textbf{Case S3a: $\V$ overdetermined.} 
				Note that the second subsystem \eqref{eq:sub2incon} is possibly inconsistent (since there may be a component $\xopt$ of in the null space of $\V^*$).
					Writing $\xopt = \V \bopt + \br_V$ such that $\br_V \in null(\V^*)$, we have
				    \begin{align*}
				     	\bopt &= (\V^*\V)^{-1} \V^* \left((\U^* \U)^{-1}\U^* (\y - \br_1 - \br_2) - \br_V \right)\\
				     			  &= (\V^*\V)^{-1} \V^* (\U^* \U)^{-1}\U^* (\y - \br_2) \\
				     			  &= \betaopt - (\V^*\V)^{-1} \V^* (\U^* \U)^{-1}\U^*\br_2 \\
				     			  &\neq \betaopt.
				   \end{align*}
				   Similar to Scenario S2, if subsystem \eqref{eq:sub2incon} is indeed consistent then $\br_V = 0$ and the above calculation still carries through.
				   \textit{Therefore, in this case, we do not expect to find the optimal solution to \eqref{eq:fullsys}}. 
				   
				   \item \textbf{Case S3b: $\V$ underdetermined.} 
					In this case, $\br_2 = \textbf{0} $ and solving \eqref{eq:sub1incon} and \eqref{eq:sub2incon} obtains the optimal solution to the full system since
					\begin{align*}
						\V \bopt & = \V \V^*(\V \V^*)^{-1} (\U^* \U)^{-1}\U^*(\y - \br_1) \\ 
							      & = \V \V^*(\V \V^*)^{-1} (\U^* \U)^{-1}\U^* \y  \\ 
						    	  & = \V \V^*(\V \V^*)^{-1} (\U^* \U)^{-1}\U^* \U \V \betaopt \\ 
						 	     & =\V \V^*(\V \V^*)^{-1}  \V \betaopt \\ 
						    	 & = \V \betaopt  
					\end{align*}				
					Following the same argument as in Scenario S1 when $\V$ is underdetermined, we reach the conclusion that $\bopt = \betaopt$. \textit{Thus, in this case our methods have the potential to solve the full system.}\\
				
			\end{itemize}		
	\end{itemize}

	These three scenarios fully explain the shading in Table~\ref{tab:underover}. 
	The focus of the remainder of this paper will be the case in which $k < m,n$ (i.e. left column of Table~\ref{tab:underover}) since, as mentioned, it is  practically the most relevant setting.


\section{Methods and Main Results} 
	\label{sec:main}
	
	Our approach intertwines two iterative methods to solve subsystem~\eqref{eq:subsys1} followed by subsystem~\eqref{eq:subsys2}. 
	For the consistent setting, we propose Algorithm~\ref{alg:rkrk} which uses an iterate of RK on \eqref{eq:subsys1} intertwined with an iterate of RK to solve \eqref{eq:subsys2}. 
	For the inconsistent setting, we propose using REK to solve subsystem \eqref{eq:sub1incon} followed by RK to solve subsystem \eqref{eq:sub2incon} as shown in Algorithm~\ref{alg:rekrk}. We view the latter method as a more practical approach, and the former as interesting from a theoretical point of view.
	Recall that $\xcur^p$ is the $p^{th}$ element in the vector $\xcur$.
	Standard stopping criteria include terminating when the difference in the iterates is small or when the residual is less than a predetermined tolerance. To avoid adding complexity to the algorithm, the residual should be computed and checked approximately every $m$ iterations.
		We propose an approach that interlaces solving subsystems \eqref{eq:subsys1} and \eqref{eq:subsys2}; this has a couple of advantages over solving each subsystem separately.
	First, if we are given some tolerance $\epsilon$ that we allow on the full system, it is unclear when we should stop the iterates of the first subsystem to obtain such an error --- if solving the first subsystem is terminated prematurely, the error may propagate through iterates when solving the second subsystem.
	Second, the interlacing allows for opportunities to implement these algorithms in parallel. 
	We leave the specifics of such an implementation as future work as it is outside the scope of this paper.

	\begin{algorithm}
\caption{RK-RK}
 \label{alg:rkrk}
\begin{algorithmic}
		\STATE{Input: $\U$, $\V$, $\y$}
		\WHILE{stopping criteria not reached}
				\STATE{Choose row  $\Ur$ with probability $\frac{\|\Ur\|_2^2}{\|\U\|_F^2}$ }
				\STATE{Update $\xcur := \xprev + \frac{(\y^i - \Ur \xprev)}{\|\Ur\|^2_2} (\U^i)^*$}
				\STATE{Choose row  $\Vr$ with probability $\frac{\|\Vr\|_2^2}{\|\V\|_F^2}$ }
				\STATE{Update $\bcur := \bprev + \frac{(\xcur^p - \Vr \bprev)}{\|\Vr\|^2_2} (\V^p)^*$}
		\ENDWHILE
\end{algorithmic}
\end{algorithm}

	\begin{algorithm}
	\caption{REK-RK}
	\label{alg:rekrk}
		\begin{algorithmic}
		\STATE{Input: $\U$, $\V$, $\y$}
		\WHILE{stopping criteria not reached}
				\STATE{Choose row  $\Ur$ with probability $\frac{\|\Ur\|_2^2}{\|\U\|_F^2}$ }
				\STATE{Choose column  $\Uc$ with probability $\frac{\|\Uc\|_2^2}{\|\U\|_F^2}$ }
				\STATE{Update $\z_{t} := \z_{t-1} - \frac{\Uc^* \z_{t-1} }{\|\Uc\|_2^2}\Uc$}
				\STATE{Update $\xcur := \xprev + \frac{(\y^i - \z^i_t + \Ur \xprev)}{\|\Ur\|^2_2} (\U^i)^*$}
				\STATE{Choose row  $\Vr$ with probability $\frac{\|\Vr\|_2^2}{\|\V\|_F^2}$ }
				\STATE{Update $\bcur := \bprev + \frac{(\xcur^p - \Vr \bprev)}{\|\Vr\|^2_2} (\V^p)^*$}
		\ENDWHILE
	    \end{algorithmic}
	\end{algorithm}

	\subsection{Main result}
	
	Our main result shows that Algorithm~\ref{alg:rkrk} and Algorithm~\ref{alg:rekrk} converge linearly to the desired solution. The convergence rate, as expected, is a function of the conditioning of the subsystems, and hence we introduce the following notation. Here and throughout, for any matrix $\A$ we write 
	\begin{align}
	\alpha_A &~:=~ 1 - \frac{\sigma_{\min}^2(\A)}{\|\A\|^2_F} ~, 	\label{eq:rate} \\
	  \conda &~:=~ \frac{\sigma_{\max}^2(\A)}{\sigma_{\min}^2(\A)}  ~,	 \label{eq:cond} 
	  \end{align} 
	   where $\sigma_{\min}^2(\A)$ is the smallest non-zero singular value of $\A$, and $\conda$ is the squared condition number of $\A$.
	  Recall that the \textit{optimal solution} to a system is either the least-norm, unique, or least-squares solution depending on whether the system is underdetermined, overdetermined consistent, or overdetermined inconsistent, respectively.

		\begin{theorem} Let $\X$ be low rank, $\X = \U \V$ such that $\U \in \mathbb{C}^{m \times k}$ and $\V \in \mathbb{C}^{k \times n}$ are full rank, and the systems $\X \bbeta = \y$, $\U \x = \y$, and $\V \bb = \x$ have optimal solutions $\betaopt$, $\xopt$, and $\bopt$ respectively, and $\au$, $\av, \kappa^2_U$ are as defined in \eqref{eq:rate} and \eqref{eq:cond}. Setting $\bb_0 = \textbf{0}$ and assuming $k < m,n$, we have
			\begin{enumerate}[(a)]
				\item  if  $\X \bbeta = \y$ is consistent, then $\bopt=\betaopt$ and Algorithm~\ref{alg:rkrk} converges with expected error
					\begin{equation*}
						\E \| \bcur - \betaopt \|^2 \leq 
						\begin{cases}
						\av^t \| \bopt \|^2 +  (1-\gamma_1)^{-1}\alpha_{\text{max}}^{t} \frac{\| \xopt \|^2}{\| \V \|^2_F}, \,\,\,  \text{ if } \au \neq \av \\ 
						\av^t \| \bopt \|^2 + t \alpha_{\text{max}}^t \frac{\| \xopt \|^2}{\| \V \|^2_F} ,\,\,\, \,\,\,  \text{ else} 
						\end{cases} 
					\end{equation*}
					where $\alpha_{\text{max}} = \max\{\au, \av\}$ and $\gamma_1 = \min \{ \frac{\au}{\av}, \frac{\av}{\au} \}$.
				\item if $\X \bbeta = \y$ is inconsistent, then  $\bopt=\betaopt$ and Algorithm~\ref{alg:rekrk} converges with expected error
					\begin{equation*}
						\E \| \bcur - \betaopt \|^2 \leq 
						\begin{cases}
						\av^t \| \bopt \|^2 +(1-\gamma_2)^{-1} \tilde{\alpha}_{\text{max}}^{t-1} \frac{(1 + 2\kappa_U^2)\| \xopt \|^2}{\| \V \|^2_F},  \,\,\, \text{ if } \sqrt{\au} \neq \av \\ 
						\av^t \| \bopt \|^2 + t \tilde{\alpha}_{\text{max}}^{t-1} \frac{(1 + 2\kappa_U^2)\| \xopt \|^2}{\| \V \|^2_F}, \,\,\, \,\,\,\,\,\, \,\,\,\text{    else} 
						\end{cases} 
					\end{equation*}
					where $\tilde{\alpha}_{\text{max}} = \max\{\sqrt{\au}, \av\}$ and $\gamma_2 = \min \{ \frac{\sqrt{\au}}{\av}, \frac{\av}{\sqrt{\au}} \}$.
			\end{enumerate}
			\label{thm:mainthm}
		\end{theorem}
		
		\noindent{\bfseries Remarks.} \\
		{\bfseries 1.} Theorem~\ref{thm:mainthm}(a) also applies to the setting in which $\X$ is overdetermined, consistent, and $n < k < m$. In the proof of Lemma \ref{lem:btilde}, one must simply note that the bound \eqref{eq:rem1} still holds for this setting and all other steps in the proof analogously follow.\\
	{\bfseries 2.} Empirically, our experiments in the next section suggest that $\U$ and $\V$ can be substantially better conditioned than $\X$. \\
 	{\bfseries 3.} Algorithm~\ref{alg:rkrk} is interesting to discuss from a theoretical standpoint but in applications Algorithm~\ref{alg:rekrk} is more practical as linear systems are typically inconsistent. Algorithm~\ref{alg:rkrk} can be utilized in applications if error in the solution is tolerable. In particular, if $\X \bbeta = \y$ is inconsistent then Algorithm~\ref{alg:rkrk} will converge in expectation to some convergence horizon. This can be see by replacing the use of Proposition \ref{lem:rk} in the bound \eqref{eq:proof} with the convergence bound of RK on inconsistent linear systems found in Theorem 2.1 of \cite{Nee10:Randomized-Kaczmarz}.\\
 	{\bfseries 4.} While not the main focus of this paper, we briefly note here that for matrices large enough that they cannot be stored entirely in memory, there is an additional cost that must be paid in terms of moving data between the disk and RAM. In a truly large-scale implementation, the RK-RK algorithm might be more scalable than the REK-RK algorithm since RK only accesses random rows of both $\U$ and $\V$, which is efficient if both matrices are stored in row major form, but REK accesses both random rows and columns, and hence storing in either row major or column major format will be slow for one of the two operations.

\subsection{Supporting results}

		To prepare for the proof of the above theorem (the central theoretical result of the paper), we state a few supporting results which will help simplify the presentation of the proof. We begin by stating known results on the convergence of RK and REK on linear systems. 
		Let $\E_{b}$ denote the expected value taken over the choice of rows in $\V$ and $\E_{x}$ the expected value taken over the choice of rows in $\U$ and when necessary the choice of columns in $\U$.
		Also, let $\E$ denote the full expected value (over all random variables and iterations) and $\E^{t-1}$ be the expectation conditional on the first $t-1$ iterations.
		
		\begin{proposition}(\cite[Theorem 2]{SV09:Comments-Randomized}) Given a consistent linear system $\X \bbeta = \y$, the Randomized Kaczmarz algorithm, with initialization $\bbeta_0 = \textbf{0}$, as described in Section~\ref{sec:existing_rk} converges to the optimal solution $\betaopt$ with expected error
			\begin{equation*}
				\E \|\bbeta_{t} - \betaopt \|^2 \leq \ax^{t} \|\betaopt \|^2,
			\end{equation*}
			where $\ax$ is as defined in \eqref{eq:rate}.
			\label{lem:rk}
		\end{proposition}
		
		\begin{proposition}(\cite[Theorem 8]{ZF12:Randomized-Extended}) Given a linear system $\X \bbeta = \y$, the Randomized Extended Kaczmarz algorithm, with initialization $\bbeta_0 = \textbf{0}$, as described in Section~\ref{sec:existing_rk} converges to the optimal solution $\betaopt$ with expected error
			\begin{equation*}
				\E \|\bbeta_{t} - \betaopt \|^2 \leq \ax^{\lfloor t/2 \rfloor} (1 + 2 \condx) \|\betaopt\|^2,
			\end{equation*}
			where $\ax$ is as defined in \eqref{eq:rate} and $\condx$ is as defined in \eqref{eq:cond}.
			\label{lem:rek}
		\end{proposition}
		
		The proof of Theorem~\ref{thm:mainthm} builds directly on two useful lemmas. Lemma~\ref{lem:btilde} addresses the impact of intertwining the algorithms. In particular, it shows useful relationships involving $\btcur$, the RK update solving the linear system $\V \bb = \xopt$ at the $t^{th}$ iteration (with $\bprev$ as the previous estimate), and our update $\bcur$. Lemma~\ref{lem:rprk} states that conditional on the first $t-1$ iterations, we can split the norm squared error $\|\bcur - \bopt \|^2$ into two terms relating to the error from solving subsystem~\eqref{eq:subsys1} and the error from solving subsystem~\eqref{eq:subsys2}. To complete the proof of Theorem~\ref{thm:mainthm}, we bound the error from solving \eqref{eq:subsys1} depending on whether we use RK (as in Algorithm~\ref{alg:rkrk}) or REK (as in Algorithm~\ref{alg:rekrk}) then apply the law of iterated expectations to bound the error from solving \eqref{eq:subsys2}. We now state the aforementioned lemmas, and then formally prove the theorem.
		
		\begin{lemma} Let $\btcur = \bprev + \frac{ \xopt^p - \Vr \bprev  }{\|\Vr\|^2 } (\Vr)^*$. In Algorithm~\ref{alg:rkrk} and Algorithm~\ref{alg:rekrk} we have that:
		\begin{enumerate}[(a)]
		\item $\E_b^{t-1} \left\langle \bcur - \btcur, \btcur - \bopt \right\rangle = 0,$ 
		\item $ \| \btcur - \bopt \|^2 = \| \bprev - \bopt \|^2 - \| \btcur - \bprev \|^2. $
		\end{enumerate}
			\label{lem:btilde}
		\end{lemma} 
		In words, part (a) states that the difference between an RK iterate solving the exact linear system $\V \bb = \xopt$ and our RK iterate (which solves the linear system resulting from intertwining $\V \bb = \xcur$), is orthogonal to $\btcur - \bopt$. This will come in handy in Lemma~\ref{lem:rprk}. Part (b) is a Pythagoras-style statement, which follows from well-known orthogonality properties of RK updates, included here for simplicity and completeness.
		\begin{proof}

		To prove statement (b), we note that $(\btcur - \bprev)$ is parallel to $ \Vr$ and $(\btcur - \bopt)$ is perpendicular to $\Vr$ since $ \Vr(\btcur - \bopt) = \Vr(\bprev + \frac{ \xopt^p - \Vr \bprev  }{\|\Vr\|^2 } (\Vr)^* - \bopt) = \Vr \bprev + \xopt^p - \Vr \bprev - \xopt^p = \textbf{0}$. We apply the Pythagorean Theorem to obtain the desired result.

		We prove statement (a) by direct substitution and expansion, as follows:
		\begin{align}
			\E_b^{t-1} &\left\langle \bcur - \btcur, \btcur - \bopt \right\rangle 
			= \E_b^{t-1}  \left\langle \frac{  \xcur^{p} - \xopt^{p} }{\|\Vr\|^2 } (\Vr)^*, \bprev - \bopt + \frac{\xopt^{p}- \Vr \bprev  }{\|\Vr\|^2 }(\Vr)^* \right\rangle \nonumber \\
				&\stackrel{(i)}{=}  \E_b^{t-1} \left\langle  \frac{ \xcur^{p} - \xopt^{p} }{\|\Vr\|^2 } (\Vr)^*, \frac{\xopt^{p}- \Vr \bprev  }{\|\Vr\|^2 }(\Vr)^* \right\rangle  +  \E_b^{t-1} \left\langle  \frac{ \xcur^{p} - \xopt^{p} }{\|\Vr\|^2 }(\Vr)^*, \bprev - \bopt\right\rangle\nonumber \\ 
				&\stackrel{(ii)}{=} \E_b^{t-1} \frac{ (\xcur^{p} - \xopt^{p})( \xopt^{p}- \Vr \bprev )  }{\|\Vr\|^2 }  +  \left\langle  \E_b^{t-1} \frac{  \xcur^{p} - \xopt^{p} }{\|\Vr\|^2 } (\Vr)^*, \bprev - \bopt\right\rangle \nonumber \\ 
				&\stackrel{(iii)}{=} \sum_p \frac{ (\xcur^{p} - \xopt^{p})( \xopt^{p}- \Vr \bprev )  }{\|\Vr\|^2 } \frac{\|\Vr\|^2}{\|\V\|^2_F}  +  \left\langle  \sum_p \frac{  (\xcur^{p} - \xopt^{p}) (\Vr)^* }{\|\Vr\|^2 }  \frac{\|\Vr\|^2}{\|\V\|^2_F}, \bprev - \bopt\right\rangle \nonumber \\
				&= \frac{ (\xcur - \xopt)^*(\xopt- \V \bprev )  }{\|\V\|^2_F}  +  \left\langle  \frac{  \V^*(\xcur - \xopt) }{\|\V\|^2_F}, \bprev - \bopt\right\rangle  \nonumber \\
				&\stackrel{(iv)}{=} \left\langle \frac{\xcur - \xopt}{\|\V\|^2_F}  , \V(\bopt- \bprev) \right\rangle + \left\langle \frac{\xcur - \xopt}{\|\V\|^2_F}   , \V( \bprev - \bopt)\right\rangle \label{eq:rem1}\\
			   &= 0.   	\nonumber
			   		\end{align}
			Step ($i$) follows from linearity of inner products, step ($ii$) simplifies the inner product of two parallel vectors, and step ($iii$) computes the expectation over all possible choices of rows of $\V$. In step ($iv$), we use the fact that for $k < m,n$, subsystem~\eqref{eq:subsys2} is always consistent (since $\V$ is underdetermined) to make the substitution $\xopt = \V \bopt$. 

				\end{proof}
 		
		\begin{lemma} In Algorithm~\ref{alg:rkrk} and Algorithm~\ref{alg:rekrk}, we can bound the expected norm squared error of $\bcur - \bopt$ as
			$$\E^{t-1}\| \bcur - \bopt \|^2  \leq  \av \| \bprev - \bopt \|^2 + \E_{x}^{t-1}\frac{ \| \xcur - \xopt \|^2}{\|\V\|^2_F }.$$ 
			\label{lem:rprk}
		\end{lemma}
We investigate the expectation of the norm squared error of $\bcur - \bopt$ conditional on the first $t-1$ iterations and over the choice of rows of $\V$. We keep $ \E_{x}^{t-1}$ in our bound as this expectation will depend on whether Algorithm~\ref{alg:rkrk} or Algorithm~\ref{alg:rekrk} is being used. 
		\begin{proof}
		
		\begin{align*}
				\E^{t-1} \| \bcur - \bopt \|^2 
				& \stackrel{}{=} \E^{t-1} \| \bcur - \bopt + \btcur - \btcur \|^2 \\ 
				&\stackrel{}{=} \E^{t-1} \| \btcur - \bopt \|^2  + \E^{t-1}\| \bcur - \btcur \|^2 + 2 \E^{t-1}  \left \langle \btcur - \bopt, \bcur - \btcur  \right \rangle \\
				& \stackrel{(iii)}{=} \E^{t-1} \| \btcur - \bopt \|^2  + \E^{t-1} \| \bcur - \btcur \|^2 \\
				& \stackrel{(iv)}{=} \E^{t-1}  \| \bprev - \bopt \|^2 - \E^{t-1}  \| \btcur - \bprev \|^2 + \E^{t-1} \|\bcur - \btcur \|^2 \\
				& \stackrel{(v)}{=} \norm{ \bprev - \bopt}^2 -  \E^{t-1} \norm {  \frac{ (\xopt^p - \Vr \bprev)  }{\|\Vr\|^2 } (\Vr)^* }^2  + \E^{t-1} \norm{ \frac{(\xcur^{p} - \xopt^{p})}{\|\Vr\|^2 }  (\Vr)^* }^2  \\
				& = \norm{ \bprev - \bopt}^2 -  \E^{t-1} \left[ \frac{ |\Vr \bopt- \Vr \bprev|^2  }{\|\Vr\|^2 }  \right]  + \E^{t-1} \left[ \frac{|\xcur^{p} - \xopt^{p}|^2}{\|\Vr\|^2 }  \right].
				\end{align*}
				Steps ($iii$) and ($iv$) are applications of Lemma~\ref{lem:btilde}(a) and Lemma~\ref{lem:btilde}(b) respectively, and step ($v$) follows from the definition of each term and simplification using the fact that $\V \bopt = \xopt$.
	 
	 Now, we evaluate the conditional expectation on the choices of rows of $\V$ to complete the proof:
	 \begin{align*}
			\E^{t-1} \| \bcur - \bopt \|^2 
			& \stackrel{(vi)}{=} \norm{ \bprev - \bopt}^2 -  \E_b^{t-1} \left[ \frac{ |\Vr \bopt- \Vr \bprev|^2  }{\|\Vr\|^2 }  \right]  + \E_{x}^{t-1} \E_b^{t-1} \left[ \frac{|\xcur^{p} - \xopt^p |^2}{\|\Vr\|^2 }   \right] \\
		& \stackrel{}{=} \norm{ \bprev - \xopt}^2 -  \sum_{p = 1}^k \frac{ |\Vr \bopt- \Vr \bprev|^2  }{\|\Vr\|^2 } \frac{\|\Vr \|^2}{\|\V\|^2_F}  + \E_{x}^{t-1} \sum_{p = 1}^k \frac{|\xcur^p - \xopt^p |^2}{\|\Vr\|^2 } \frac{\|\Vr \|^2}{\|\V\|^2_F} \\
			& \stackrel{}{=} \norm{ \bprev - \bopt}^2 - \frac{ \|\V \bopt - \V \bprev \|^2  }{\|\V\|^2_F } + \E_{x}^{t-1} \left[ \frac{\| \xcur - \xopt \|^2}{\|\V\|^2_F }  \right] \\
			& \stackrel{(vii)}{\leq} \norm{ \bprev - \bopt}^2 - \frac{ \sigma_{\min}^2(\V) \|\bprev -  \bopt  \|^2  }{\|\V\|^2_F } + \E_{x}^{t-1} \left[ \frac{\| \xcur- \xopt \|^2}{\|\V\|^2_F }   \right] \\
			& = \av \| \bprev - \bopt \|^2 + \E_{x}^{t-1}\left[ \frac{\| \xcur - \xopt \|^2}{\|\V\|^2_F }  \right] .
		\end{align*}
		In step ($vi$), we use iterated expectations to split the expected value $\E^{t-1} = \E_{x}^{t-1} \E_b^{t-1}$. Step $(vii)$ uses the fact that $\| \V (\bprev - \bopt) \|^2 \geq \sigma_{\min}^2(\V) \| \bprev - \bopt \|^2 $ since $\bprev - \bopt$ are in the row span of $\V$ for all $t$. We simplify and obtain the desired bound.
	 	 \end{proof}
	
\subsection{Proof of main result}

	We now have all the ingredients we need to prove Theorem~\ref{thm:mainthm}, which we now proceed to below.
	
	\begin{proof}[Proof of Theorem~\ref{thm:mainthm}]
The fact that $\bopt=\betaopt$ was already argued in scenarios S1 and S3(b) in the previous section, so we do not reproduce its argument here. Given this fact, to prove Theorem~\ref{thm:mainthm}, we only need to invoke the statement of Lemma~\ref{lem:rprk} and bound the term $\E_{x}^{t-1} \| \xcur - \xopt \|^2$ using Proposition~\ref{lem:rk} or Proposition~\ref{lem:rek} depending on whether we are using Algorithm~\ref{alg:rkrk} or Algorithm~\ref{alg:rekrk}, respectively. 
		\begin{enumerate}[(a)] 
			\item For Algorithm~\ref{alg:rkrk}, plugging Proposition~\ref{lem:rk} into the statement of Lemma~\ref{lem:rprk} yields
				\begin{align}
					\E^{t-1} \| \bcur - \bopt \|^2 
					& \stackrel{}{\leq} \av \| \bprev - \bopt \|^2 + \au^t \frac{\| \xopt \|^2}{\|\V\|^2_F }. \label{eq:proof}
					\end{align}
					Let $\alpha_{\text{max}} = \max \{\alpha_U, \alpha_V \}$, $\gamma_1 = \min\{ \frac{\alpha_U}{\alpha_V}, \frac{\alpha_V}{\alpha_U} \}$, and note that $\gamma_1 < 1$ if $\alpha_U \neq \alpha_V$. Taking expectations over the randomness from the first $t-1$ iterations and using the Law of Iterated Expectation, we have
					\begin{align*} 
					\E  \| \bcur - \bopt \|^2 & \stackrel{}{\leq}  \av^t \| \bopt \|^2 + \frac{\| \xopt \|^2}{\|\V\|^2_F } \sum_{h=0}^{t-1} \au^{t-h} \av^h \\
 					& \stackrel{}{=}  \av^t \| \bopt \|^2 + \au \frac{ \| \xopt \|^2}{\|\V\|^2_F } \sum_{h=0}^{t-1} \au^{t-1-h} \av^h \\
					& \stackrel{(i)}{=}  \av^t \| \bopt \|^2 + \alpha_{\text{max}}^{t-1} \au \frac{ \| \xopt \|^2}{\|\V\|^2_F } \sum_{h=0}^{t-1} \gamma_1^h\\
					& \stackrel{(ii)}{\leq}  \av^t \| \bopt \|^2 + \alpha_{\text{max}}^{t} \frac{ \| \xopt \|^2}{\|\V\|^2_F } \frac{1}{1-\gamma_1}.
				\end{align*}
				where step (i) uses the fact that the summation is symmetric with respect to $\alpha_U$ and $\alpha_V$. Step (ii) uses the fact that $\alpha_{\text{max}}^{t-1} \au = \alpha_{\text{max}}^{t}$ if $\alpha_{\text{max}} = \au$ and $\alpha_{\text{max}}^{t-1} \au < \alpha_{\text{max}}^{t}$ if $\alpha_{\text{max}} = \av$. Now, if $\alpha_U = \alpha_V = \alpha_{\text{max}}$ then we have
					\begin{align*} 
					\E  \| \bcur - \bopt \|^2 & \stackrel{}{\leq}  \av^t \| \bopt \|^2 + \frac{\| \xopt \|^2}{\|\V\|^2_F } \sum_{h=0}^{t-1} \alpha_{\text{max}}^{t-h} \alpha_{\text{max}}^h \\
					& \stackrel{(ii)}{=}  \av^t \| \bopt \|^2 + t\alpha_{\text{max}}^t \frac{\| \xopt \|^2}{\|\V\|^2_F },
					\end{align*}
					where the second term of $(ii)$ approaches 0 as $t \rightarrow \infty$ since $\alpha_{\text{max}}< 1$.	
			\item For Algorithm~\ref{alg:rekrk}, plugging Proposition~\ref{lem:rek} into the statement of Lemma~\ref{lem:rprk} yields
				\begin{align*}
					\E^{t-1} \| \bcur - \bopt \|^2 
					& \stackrel{}{\leq} \av \| \bprev - \bopt \|^2 + (1+2\condu)\au^{\lfloor t/2 \rfloor} \frac{\| \xopt \|^2}{\|\V\|^2_F }.
					\end{align*}
					Taking expectations over the remaining randomness, and using the fact that $\au^{ \lfloor \frac{t-h}{2} \rfloor} \leq \au^{\frac{t-h-1}{2}} $ since $\au < 1$, we have
					\begin{align*}
					\E  \| \bcur - \bopt \|^2  
					& \leq  \av^t \| \bopt \|^2 + (1+2\condu) \frac{\| \xopt \|^2}{\|\V\|^2_F } \sum_{h=0}^{t-1} \au^{\lfloor \frac{t-h}{2} \rfloor}  \av^h \\
					& \leq  \av^t \| \bopt \|^2 + (1+2\condu) \frac{\| \xopt \|^2}{\|\V\|^2_F } \sum_{h=0}^{t-1} \au^{\frac{t-h-1}{2}} \av^h \\
					& =  \av^t \| \bopt \|^2 + (1+2\condu) \frac{\| \xopt \|^2}{\|\V\|^2_F } \sum_{h=0}^{t-1} {\sqrt{\au}}^{t-1-h}\av^h .
					\end{align*}
					Using the same techniques as in (a), let $\tilde{\alpha}_{\text{max}} = \max \{ \av, \sqrt{\au} \}$ and $\gamma_2 = \min\{ \frac{\av}{\sqrt{\au}} , \frac{\sqrt{\au}}{\av} \}$ and again noting $\gamma_2 < 1$ when $\av \neq \sqrt{\au}$, we can write:
					\begin{align*}
					\E  \| \bcur - \bopt \|^2  
					& \leq  \av^t \| \bopt \|^2 + (1+2\condu) \frac{\| \xopt \|^2}{\|\V\|^2_F } \sum_{h=0}^{t-1} {\sqrt{\au}}^{t-1-h} \av^h \\	
					& =  \av^t \| \bopt \|^2 + (1+2\condu) \tilde{\alpha}_{\text{max}}^{t-1} \frac{\| \xopt \|^2}{\|\V\|^2_F } \sum_{h=0}^{t-1} \gamma_2^h \\
					& \leq  \av^t \| \bopt \|^2 + \tilde{\alpha}_{\text{max}}^{t-1} (1+2\condu) \frac{\| \xopt \|^2}{\|\V\|^2_F }\frac{1}{1-\gamma_2 } . 		
				\end{align*}
				When $\av = \sqrt{\au} = \tilde{\alpha}_{\text{max}}$, 
				\begin{align*}
					\E  \| \bcur - \bopt \|^2  
					& \leq  \av^t \| \bopt \|^2 + (1+2\condu) \frac{\| \xopt \|^2}{\|\V\|^2_F } \sum_{h=0}^{t-1}  {\tilde{\alpha}}_{\text{max}}^{t-1-h} \tilde{\alpha}_{\text{max}}^h \\	
					& \stackrel{(i)}{=}  \av^t \| \bopt \|^2 + t\tilde{\alpha}_{\text{max}}^{t-1}(1+2\condu)  \frac{\| \xopt \|^2}{\|\V\|^2_F },
				\end{align*}
				where the second term of $(i)$ goes to 0 as $t$ goes to infinity.
				
		\end{enumerate}This concludes the proof of the theorem.
	\end{proof}


\section{Experiments}\label{sec:exp}
In this section we discuss experiments done on both simulated and real data using different algorithms in different settings. 
The naming convention for the remainder of the paper will be to refer to ALG1-ALG2 as an interlaced algorithm where ALG1 is the algorithm iterate used to solve subsystem~\eqref{eq:subsys1} and ALG2 is the algorithm used to solve subsystem~\eqref{eq:subsys2}. When an algorithm's name is used alone, we imply applying the algorithm on the full system \eqref{eq:fullsys}.

\begin{figure}[h]
\includegraphics[scale=.35]{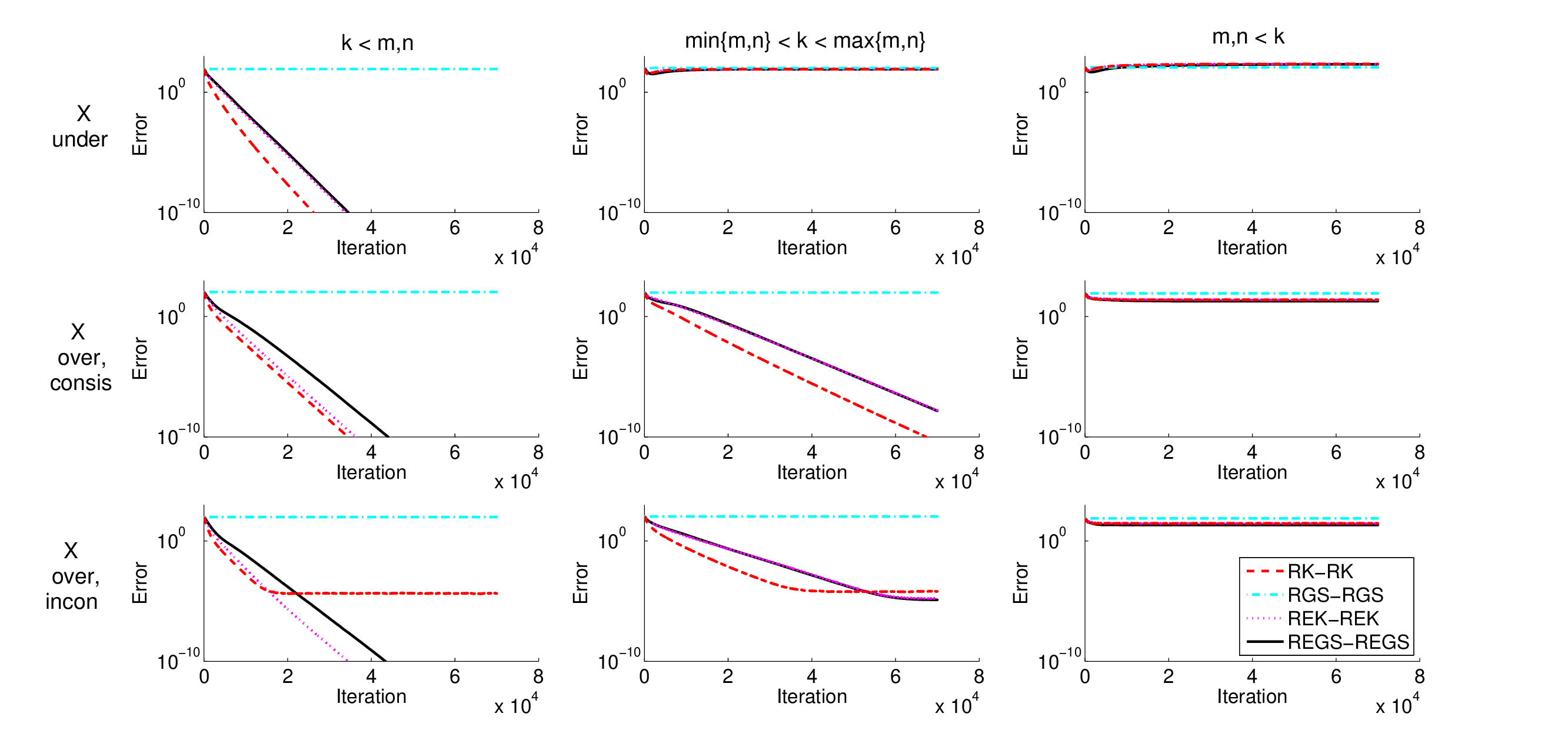}
\caption{This figure shows a summary of convergence for all methods under the variation of possible settings. In the right column, we have that none of the methods convergence. In the middle column, when $\X$ is underdetermined or inconsistent none of the methods converge either. In all other variants, the convergence depends on the general behavior of the standard (RK, RGS) algorithms.}
\label{fig:summary}
\end{figure}

In Figure~\ref{fig:summary} we show our first set of experiments. Entries of $\U$, $\V$, and $\bbeta$ are drawn from a standard Gaussian distribution. We set $\X = \U \V$ and $\y = \X \bbeta$ if $\X$ is consistent and $\y = \X \bbeta + \br$ where $\br \in null (\X^*)$ (computed in Matlab using \texttt{null()} function) if $\X$ is inconsistent. In this first set of experiments, $m,n,k \in \{ 100, \, 150, \, 200 \}$ depending on the desired size of $k$ with respect to the over or underdetermined-ness of $\X$. For example, if $k < m,n$ and $\X$ is overdetermined then $k = 100$, $m = 200$, and $n = 150$. The plots show iteration vs $\ell_2$-error, $\| \bcur - \betaopt \|^2$, of each method averaged over 40 runs and allowing each algorithm to run $7 \times 10^4$ iterations. The layout of Figure~\ref{fig:summary} is exactly as in Table~\ref{tab:underover}. For each row, we have a different setting for $\X$ and for each column, we vary the size of $k$ depending on the size of $\X$. Looking at the overall trends, we see that when $k < m,n$ and when $\X$ is overdetermined, consistent and $n < k < m$, there is a method that obtains the optimal solution for the system. These results align with the expectations set in Table~\ref{tab:underover}. Looking at each individual subplot, we also find what one would expect according to Table~\ref{tab:rkrgs}. In other words, if $\U$ or $\V$ is in one of the settings where RK or RGS are expected to fail then RK-RK or RGS-RGS fail as well.

When $\X$ is overdetermined, inconsistent and $k < m,n$ we have that $\V$ is underdetermined. In this case, we don't need to interlace iterates of REK and REK together. To work on an underdetermined system, using RK is enough to find the optimal solution of that subsystem. This motivated interlacing iterates of RK with REK. Figure~\ref{fig:incons} has the same set up as discussed in the previous experiment with the exception of using larger random matrices with $\X: 1200 \times 750$ and $k = 500$. In Figure~\ref{fig:incona} we plot iteration vs $\ell_2$-error and in Figure~\ref{fig:inconb} we plot FLOPS vs $\ell_2$-error. The errors are averaged over 40 runs, with shaded regions representing error within two standard deviations from the mean. Note that Algorithm \ref{alg:rekrk} performs excellently in practice, better than our theoretical upper bound as computed in Theorem \ref{thm:mainthm}. In this experiment, we see that REK-RK and REK-REK perform comparably in error and that REK-RK is more efficient in FLOPS.

\begin{figure*}[h]
    \centering
    \begin{subfigure}[b]{0.45\textwidth}
        \centering
        \includegraphics[width=\textwidth]{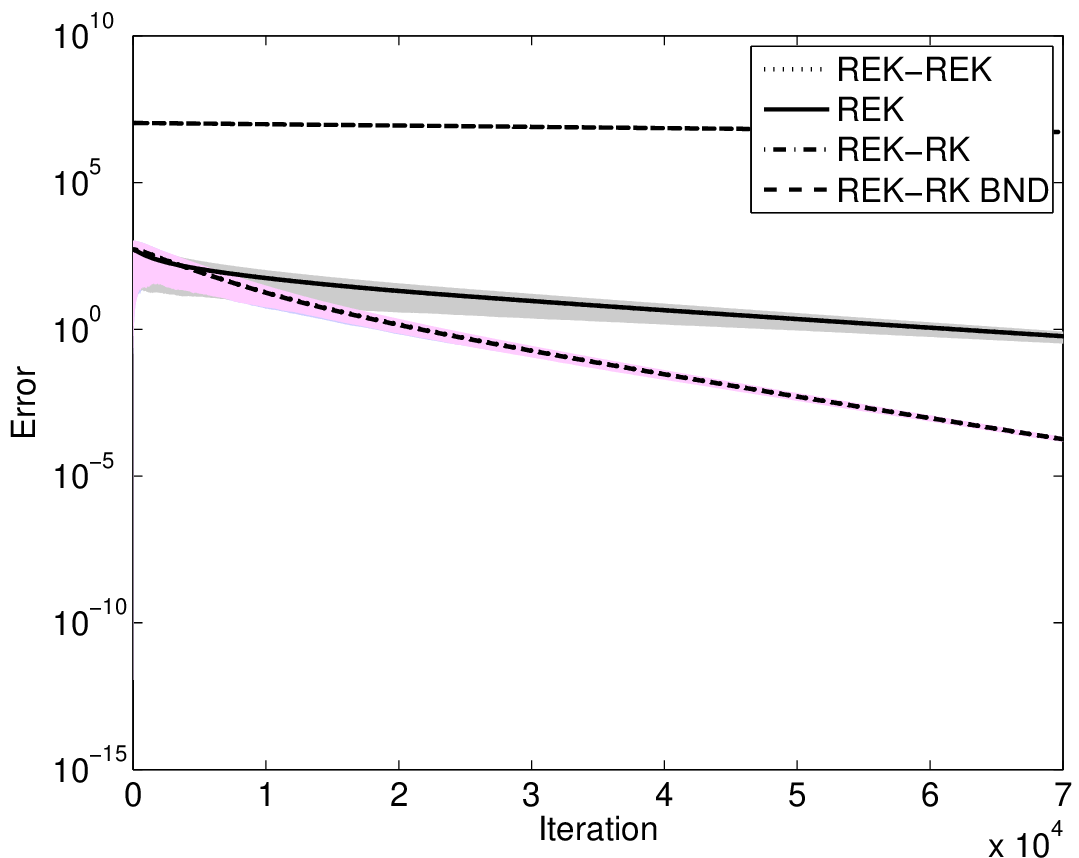}
        \caption{ }
        \label{fig:incona}
    \end{subfigure}%
    ~ 
    \begin{subfigure}[b]{0.45\textwidth}
        \centering
        \includegraphics[width=\textwidth]{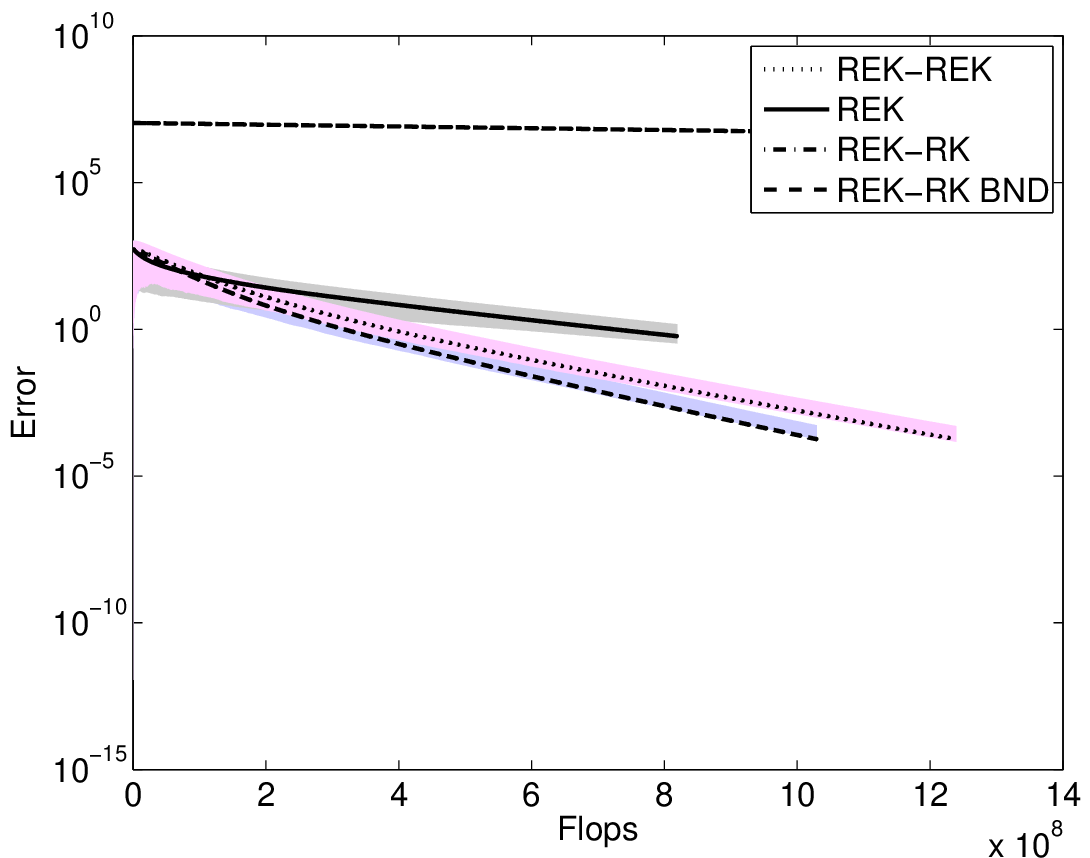}
        \caption{ }
        \label{fig:inconb}
    \end{subfigure}
    \caption{When $\X$ is overdetermined, inconsistent and $k < m,n$ we propose interlacing iterates of REK  and RK to solve subsystems \eqref{eq:subsys1} and \eqref{eq:subsys2}. This figure demonstrates the advantage of using REK-RK on an inconsistent system as opposed to REK-REK, the former needing less FLOPS to achieve the same accuracy.} 
    \label{fig:incons}
\end{figure*}

In addition to simulated experiments, we also show the usefulness of these algorithms on real world data sets on wine quality, bike rental data, and Yelp reviews. In all following experiments, we plot the average $\ell_2$-error at the $t^{th}$ iteration over 40 runs and shaded regions representing the $\ell_2$-error within two standard deviations. In addition to empirical performance, we also plot the theoretical convergence bound derived in Theorem \ref{thm:mainthm} (labeled ``BND" in the legends). From these experiments, it is clear that the algorithms perform even better in practice than the worst-case theoretical upper bound. The data sets on wine quality and bike rental data are obtained from the UCI Machine Learning Repository \cite{Lichman:2013}. The wine data set is a sample of $m=1599$ red wines with $n = 11$ physio-chemical properties of each wine. We choose $k = 5$ and compute $\U$ and $\V$ using Matlab's \texttt{nnmf()} function for nonnegative matrix factorization (recall the motivations from the first section). Figure~\ref{fig:wine} shows the results from this experiment. The conditioning of $\X$, $\U$, and $\V$ are $\condx = 2.46 \times 10^3$, $\condu = 25.96$, and $\condv = 4.20$ respectively. We plot the $\ell_2$-error averaged over 40 runs. Since $\X$ has such a large condition number, this impacts the convergence of REK on $\X$ negatively as shown by the seemingly horizontal line (the error is actually decreasing, but incredibly slowly). We also see that REK-RK and REK-REK are working comparably and significantly faster than REK alone. This can be explained by the better conditioning on $\U$ and $\V$.

\begin{figure*}[h]
    \centering
    \begin{subfigure}[b]{0.45\textwidth}
        \centering
        \includegraphics[width=\textwidth]{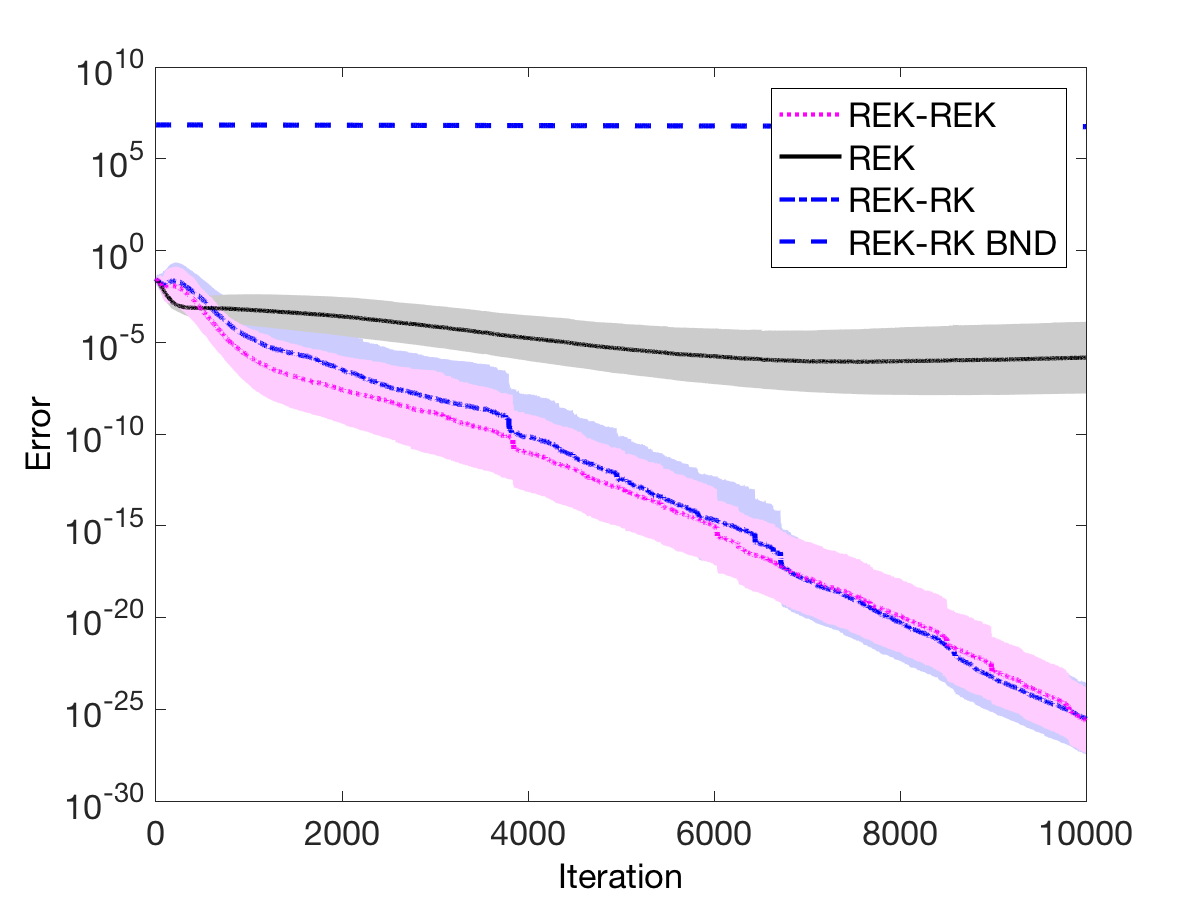}
        \caption{}
        \label{fig:wine}
    \end{subfigure}%
    ~ 
    \begin{subfigure}[b]{0.45\textwidth}
        \centering
        \includegraphics[width=\textwidth]{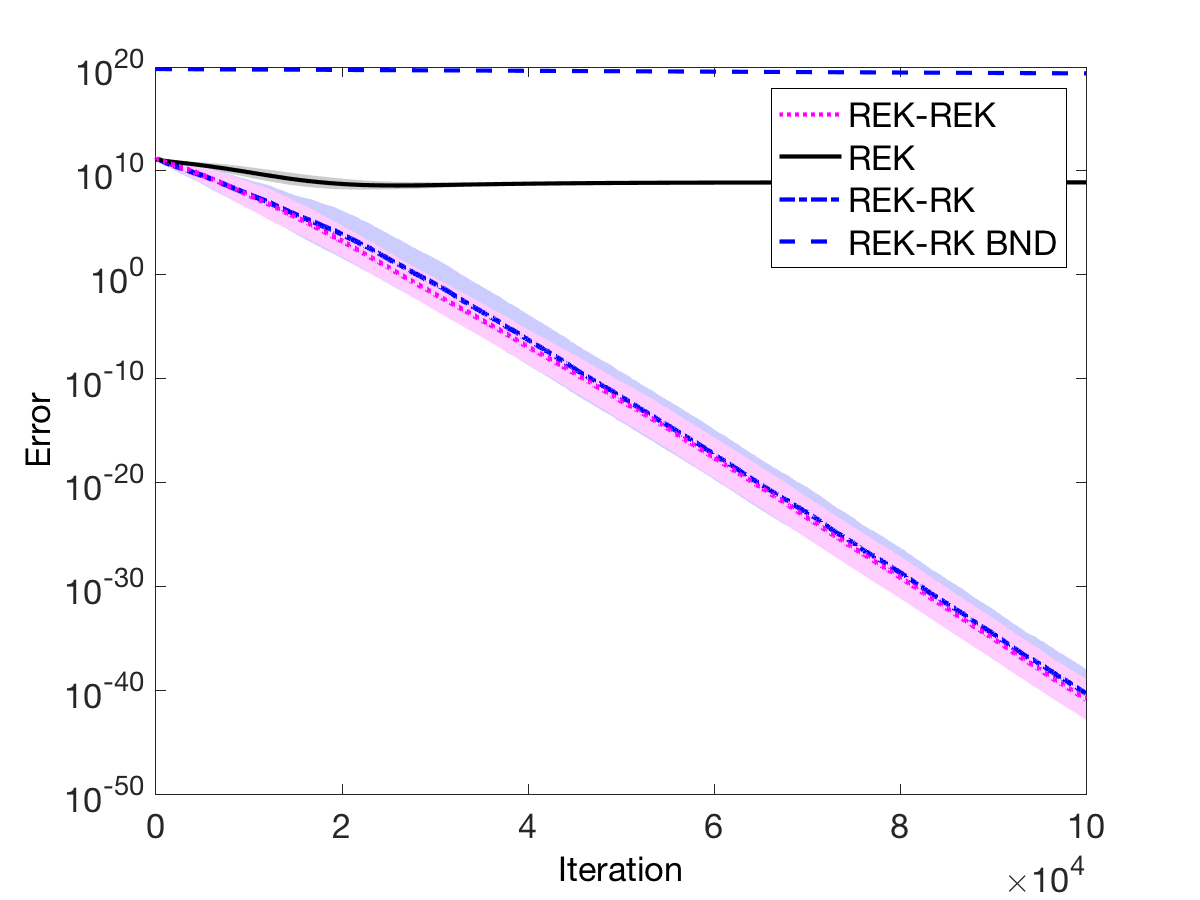}
        \caption{ }
        \label{fig:bike}
    \end{subfigure}
    \caption{In these experiment, we compare the empirical performance of REK, REK-REK, and  REK-RK on real world data. Figure~\ref{fig:wine} shows the performance of these methods on the wine data set and Figure~\ref{fig:bike} shows performance on the bike data set.}
    \label{fig:realdata}
\end{figure*}

The bike data set contains hourly counts of rental bikes in a bike share system. The data sets contains date as well as weather and seasonal data. There are $m = 17379$ samples and $n = 9$ attributes per sample. We choose $k = 8$ and compute $\U$ and $\V$ in the same way as with the wine data set. Figure~\ref{fig:bike} shows the results from this experiment. The conditioning of $\X$, $\U$, and $\V$ are $\condx = 94.27$, $\condu = 54.91$, and $\condv = 2.99$ respectively. Similar to Figure~\ref{fig:wine}, we see that the convergence of REK suffers from the poorly conditioned matrix $\X$. We also see again that REK-REK and REK-RK behave similarly and outperform REK.

To show the advantage of our algorithms on large systems, we create extremely large standard Gaussian matrices $\U: 10^6 \times 10^3$ and $\V: 10^3 \times 10^4$. These matrices are so large that the matrix product $\U \V$ cannot be computed in Matlab due to memory constraints. These results are shown in Figure~\ref{fig:bigrealdata}. We see that without needing to do the matrix computation, we are still able to find the solution to the linear system $\U \V \bbeta = \y$.

\begin{figure*}[h]
    \centering
    \begin{subfigure}[b]{0.45\textwidth}
      \includegraphics[width=\textwidth]{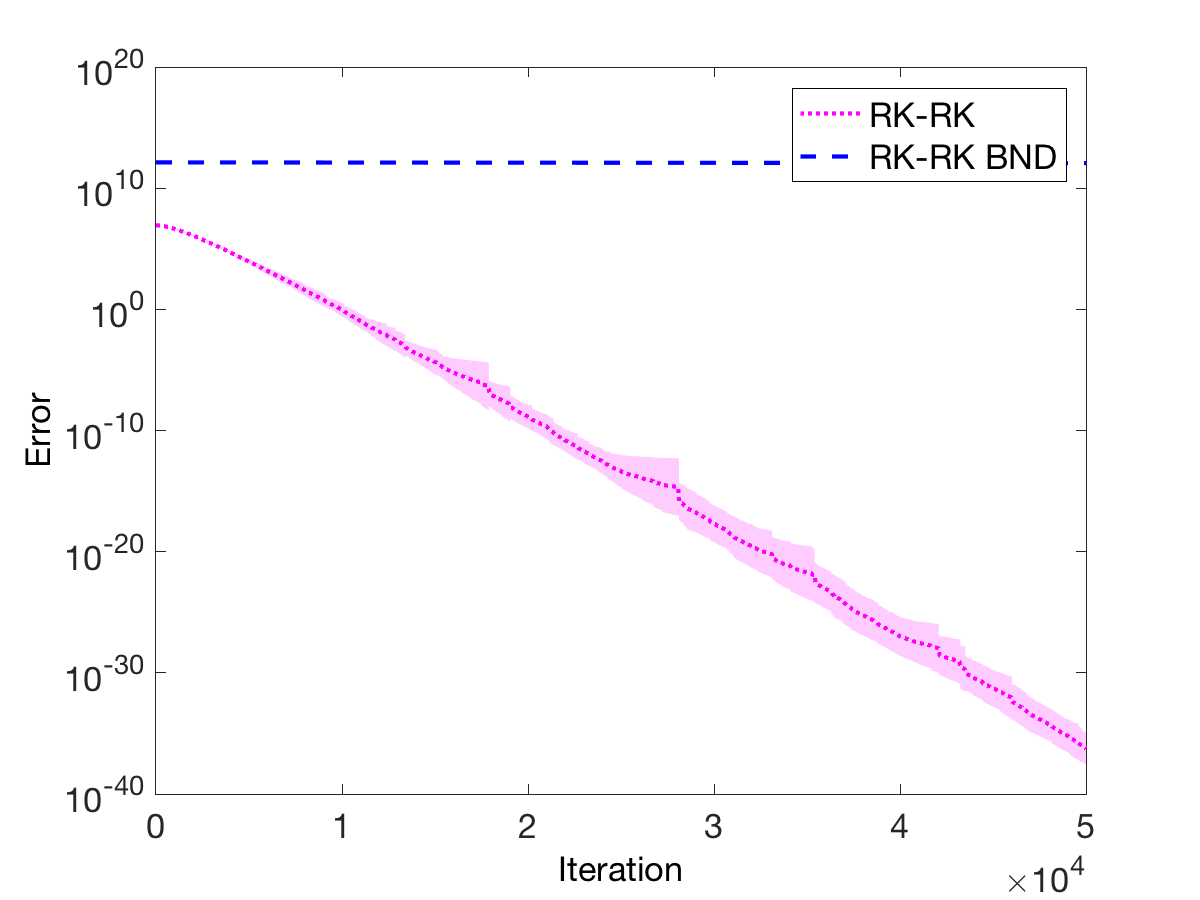}
	\caption{}
	\label{fig:bigexp}
    \end{subfigure}%
    ~ 
    \begin{subfigure}[b]{0.45\textwidth}
        \centering
        \includegraphics[width=\textwidth]{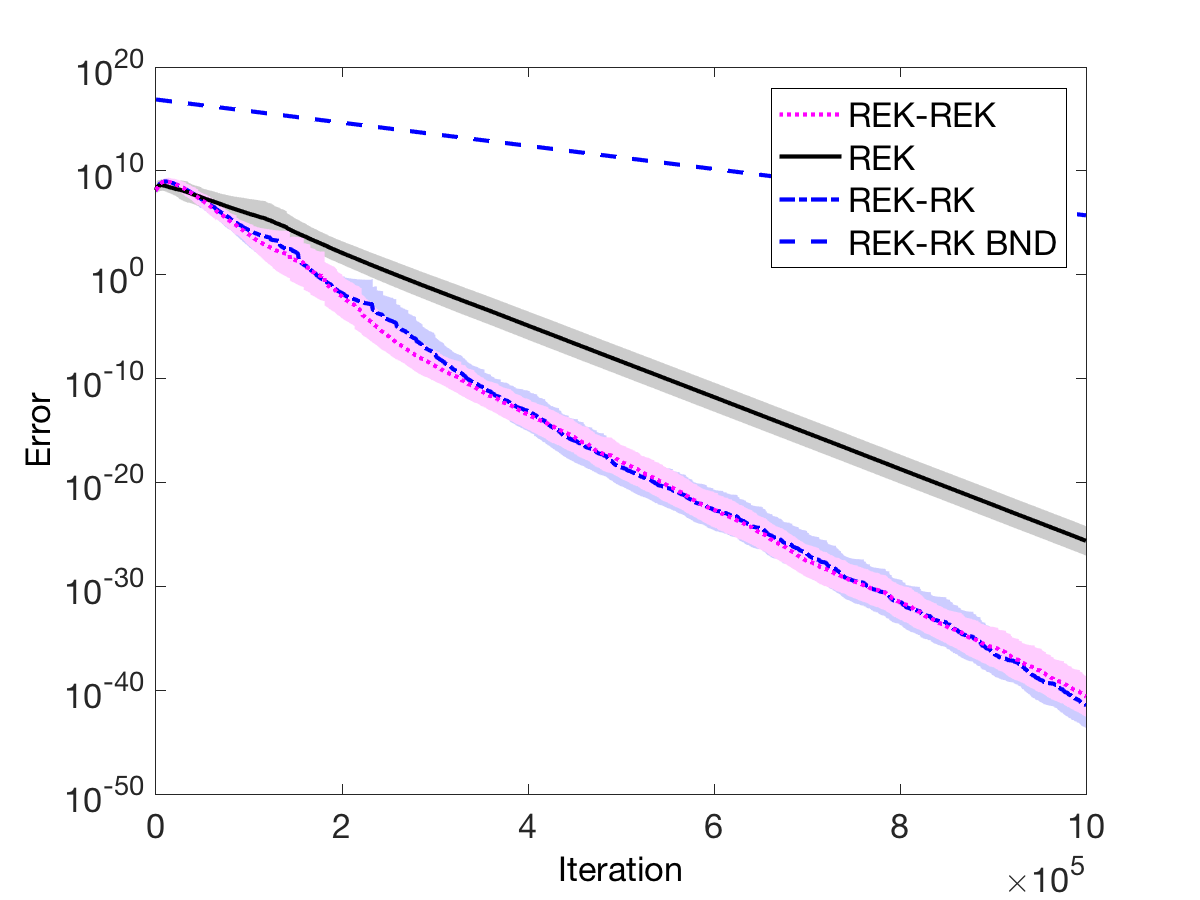}
        \caption{}
        \label{fig:yelpflop}
    \end{subfigure}
    \caption{We compare the performance of REK, REK-REK, and REK-RK on extremely large datasets. Figure~\ref{fig:bigexp} shows results from an experiment that pushes the limits of memory in Matlab. Note that in this experiment, we cannot perform RK on the full system as the matrix product requires too much memory and cannot be formed in Matlab. Figure~\ref{fig:yelpflop} shows the performance of our method on the Yelp dataset. }
    \label{fig:bigrealdata}
\end{figure*}

Lastly, we present the performance of our methods on a large real world data set. We use the Yelp challenge data set \cite{yelp}. In our setting, we let $\X: 10^5 \times 10^4$ be a document term frequency matrix where each row represents a Yelp review and each column represents a word feature. The elements of $\X$ contain the frequency at which the word is used in the review. We only use a subset of the amount of data available due Matlab memory constraints. Here, $\y$ is a vector that represents the number of stars a review received. We choose $k = 5000$. Figure~\ref{fig:yelpflop} shows the results from this experiment using REK, REK-REK, and REK-RK. The conditioning of $\X$, $\U$, and $\V$ are  $\condx = 127.3592$, $\condu = 24.274$, and $\condv = 19.096$ respectively. In this large real world data set, we can again see the usefulness of our proposed methods when we are given $\X = \U \V$.

These experiments complement and verify our theoretical findings. In settings which we expect to fail to obtain the least squares or least norm solutions, our experiments show that they do indeed fail. Additionally, where we expect that the optimal solution is obtainable, the experiments 
show the proposed methods can obtain such solutions and in many instances outperform the original algorithm on the full system. We see that empirically, subsystems are better conditioned than full systems, thus explaining their better performance.


\section{Conclusion}
\label{sec:conclu}

We have proposed two methods interlacing Kaczmarz updates to solve factored systems.  For large-scale applications in which the system is stored in factored form for efficiency or the factorization arises naturally, our methods allow one to solve the system without the need to perform the large-scale matrix product first.  Our main result proves that our methods provide linear convergence in expectation to the (least-squares or least-norm) solution of (overdetermined or underdetermined) linear systems.  Our experiments support these results, and show that our methods provide significant computational advantages for factored systems.  The interlaced structure of our methods suggests they can be implemented in parallel which would lead to even further computational gains. We leave such details for future work.  Additional future work includes the design and analysis of methods that converge to the solution in the settings not covered in this paper, i.e. the gray cells of Table~\ref{tab:underover}.  Although its practical implications are not immediately clear to us, these may still be of theoretical interest.


\section*{Acknowledgments}
Needell was partially supported by NSF CAREER grant $\#1348721$, and the Alfred P. Sloan Fellowship. Ma was supported in part by NSF CAREER grant $\#1348721$, the CSRC Intellisis Fellowship, and the Edison International Scholarship. The authors would like to thank Wutao Si for pointing out a flaw in the original proof of Theorem 1.  In addition, the authors also thank the Institute of Pure and Applied Mathematics (IPAM) where this collaboration started. 
\bibliographystyle{abbrvnat}
\bibliography{rk}
\end{document}